\documentclass[10pt]{amsart}

\usepackage{amsmath}
\usepackage{amsthm}
\usepackage{amsfonts}
\usepackage{amssymb,latexsym}
\usepackage[all]{xy}
\usepackage{graphicx}
\usepackage[mathscr]{eucal}
\usepackage{verbatim}
\usepackage{hyperref}

\theoremstyle{plain}
    \newtheorem{thm}{Theorem}[section]

    \newtheorem{cor}[thm]{Corollary}

\theoremstyle{definition}

\theoremstyle{remark}

\numberwithin{equation}{section}

\newcommand{\rar}{\ensuremath{\rightarrow}}
\newcommand{\lrar}{\ensuremath{\longrightarrow}}

\newcommand{\Hom}{\textup{Hom}}

\newcommand{\stmod}{\textup{stmod}}

\newcommand{\uHom}{\underline{\Hom}}

\newcommand{\lstk}[1]{\stackrel{#1}{\longrightarrow}}

\newcommand{\thick}{\textup{thick}}

\def\HHH{\operatorname{H}\nolimits}

\def\HHHH{\operatorname{\hat{H}}\nolimits}
\def\Hom{\operatorname{Hom}\nolimits}

\def\hExt{\operatorname{\widehat{Ext}}\nolimits}

\def\CE{{\mathcal{E}}}

\def\Rad{\operatorname{Rad}\nolimits}

\begin{document}

\title
{Freyd's generating hypothesis with almost split sequences}
\date{\today}

\author{Jon F. Carlson}
\address{
Department of Mathematics\\
University of Georgia\\
Athens, GA 30602.}
\email{jfc@math.uga.edu}

\author{Sunil K. Chebolu}
\address{Department of Mathematics  \\
Illinois State University \\
Normal, IL 61790}
\email{schebol@ilstu.edu}

\author{J\'{a}n Min\'{a}\v{c}}
\address{Department of Mathematics\\
University of Western Ontario\\
London, ON N6A 5B7, Canada}
\email{minac@uwo.ca}

\thanks{The first author is partially supported by a grant from the NSF and the
third  author is supported from the NSERC}

\keywords{Tate cohomology, generating hypothesis,
stable module category, ghost map, almost split sequence.}
\subjclass[2000]{Primary 20C20, 20J06; Secondary 55P42}

\begin{abstract}
Freyd's generating hypothesis for the stable module category of a
non-trivial finite group $G$ is the statement that
a map between finitely generated 
$kG$-modules that belongs to the thick subcategory generated by $k$
factors through a projective if the induced map on Tate
cohomology is trivial. In this paper we show that Freyd's 
generating hypothesis fails
for $kG$ when the Sylow $p$-subgroup of
$G$ has order at least $4$ using almost split sequences. 
By combining this with our earlier work, we obtain 
a complete answer to Freyd's generating hypothesis for 
the stable module category 
of a finite group.  We also derive some consequences of 
the generating hypothesis.
\end{abstract}

\maketitle
\thispagestyle{empty}


\section{Introduction}
The second and third authors have studied the generating
hypothesis (GH) and related questions for the stable module category 
of a finite group in a series of papers  \cite{CCM2, CCM3, CCM4, CCM} with Benson and Christensen, 
and have obtained many partial results. In this paper we give  a complete solution to
the generating hypothesis for the stable module category of a finite group.
We begin by recalling the statement of the GH. Let $G$
be a finite group and let $k$ be a field of characteristic $p$.
We will work in the stable module category $\stmod(kG)$
of $kG$. Recall that this is the tensor triangulated 
category obtained from the category of 
finitely generated left $kG$-modules by factoring out the projective modules. 
In the stable module category, $\uHom_{kG}(A, B)$ 
stands for the space of maps between modules $A$ and $B$, 
and $\Omega$ will denote the translation functor. 
Loosely speaking, the generating hypothesis claims 
that if a module $L$ generates a subcategory then the 
functor $\uHom_{kG}(\Omega^* L, -)$
detects trivial maps (maps that factor through a 
projective) in the subcategory. We make this precise 
in the case that is of interest to us, namely the 
subcategory generated by the trivial representation $k$.
Let $\thick_G(k)$ denote the thick subcategory generated by $k$. That is, 
the smallest full subcategory of $\stmod(kG)$  
that contains the trivial module $k$ and
is closed under exact triangles and direct summands. 
The Generating Hypothesis (GH) for a group ring $kG$ is the
statement that the Tate cohomology functor 
$\HHHH^*(G, -) \cong  \uHom_{kG}(\Omega^* k, - )$
\begin{eqnarray*}
 \thick_G(k) & \lrar & \HHHH^*(G,k)\text{-modules}  \\
   M & \mapsto & \HHHH^*(G, M) 
\end{eqnarray*}
detects trivial maps in $\thick_G(k)$, i.e., the Tate 
cohomology functor $\HHHH^*(G, -)$
on $\thick_G(k)$ is faithful.  Note that when $G$ 
is a $p$-group, then $\thick(k)$
is the entire stable module category $\stmod(kG)$.

It is shown in \cite{CCM3} that the GH holds for $kP$ when $P$
is a $p$-group if and only if $P$ is $C_2$ or $C_3$.
It is natural then to  conjecture that for an arbitrary finite
group $G$  the GH holds for $kG$ if and only if the Sylow $p$-subgroup of
$G$ is either $C_2$ or $C_3$.    In \cite{CCM4} this conjecture
has been proved using block theory for  groups which  have periodic cohomology.
In this paper, we  show that the GH fails for groups with
non-periodic cohomology. In fact, we show that the GH 
fails whenever the Sylow $p$-subgroup of $G$ has order at least 4.
Thus we have a  complete  solution to the Freyd's 
generating hypothesis in the stable module category:

\begin{thm} \label{maintheorem}
Let $G$ be a finite group and let $k$ be a field of
characteristic $p$ that divides the
order of $G$. Then the GH holds for $kG$  if and only
if a Sylow $p$-subgroup of $G$ is isomorphic to either $C_2$ or $C_3$.
\end{thm}

Our main tool in showing the failure of the GH is an almost 
split sequence. In contrast with our earlier counting techniques
\cite{CCM4} which only show the existence of a counter-example 
to the GH, our current treatment with almost split sequences
has the advantage that it produces an explicit and simple 
counter-example whenever the GH fails. 

For the interested reader we mention that there has been great 
interest in generating hypothesis in other triangulated categories
including the stable homotopy category of spectra ~\cite{freydGH} 
where it originated, but also in the derived categories of rings 
~\cite{keir, GH-D(R)}.


Throughout the paper $G$ will denote a finite 
group, and $k$ will be
a field of characteristic $p$ which divides the order 
of $G$. We use the standard 
facts about the stable module category of $kG$ which 
can be found in ~\cite{carlson-modulesandgroupalgebras}.

The first author thanks the Alexander von Humboldt Foundation 
for support and the RWTH in Aachen for their hospitality while
part of this paper was written.

\section{The generating hypothesis}

A \emph{ghost} in  $\thick(k)$ is a map between $kG$-modules 
in $\thick(k)$ that induces the trivial map in Tate cohomology.
Our goal is to show that there are non-trivial ghosts in $\thick(k)$ whenever
the Sylow $p$-subgroup of $G$ has order at least 4. Our 
main tool in showing the
existence of these non-trivial ghosts is an almost split 
sequence (a.k.a. Auslander-Reiten sequence) 
which we now define. 

A short exact sequence 
$$
\epsilon \colon \ \ \ 0 \lrar A \lrar B \lrar C \lrar 0
$$
of finitely generated $kG$-modules is an almost split 
sequence if $A$ and $C$ are indecomposable $kG$-modules, 
and $\epsilon$ is  a non-split sequence with the property 
that every map $M \rar C$  which is not a split epimorphism 
factors through the middle term $B$ \cite{ARS}.  It is a theorem of 
Auslander and Reiten that given any finitely 
generated indecomposable non-projective $kG$-module $C$, 
there exists a unique (up to isomorphism of short 
exact sequences) almost split sequence  terminating in $C$.  
Moreover, the first term $A$ of the 
almost split sequence ending in $C$
is shown to be isomorphic to  $\Omega^2 C$.

The theorem that we are now interested in is the following.

\begin{thm}\label{thm:newghosts}
Let $M$ and $N$ be two non-projective indecomposable
modules in $\stmod(kG)$ such that $N \not\cong \Omega^i(M)$ for any $i$.
Then there exists a non-trivial map
$\phi \colon N \longrightarrow \Omega N$ in $\stmod(kG)$
such that the induced map
$$
\xymatrix{
\phi_*: \hExt_{kG}^*(M,  N) \ar[r] & \hExt_{kG}^*(M, \Omega N)
}
$$
is the zero map.
\end{thm}

\begin{proof}
Consider the almost split sequence 
\[ 0 \lrar \Omega^2 N \lrar B \lrar N \lrar 0 \]
ending in $N$. This short exact sequence represents a distinguished triangle
\[  \Omega^2 N \lrar B \lrar N \lstk{\phi}  \Omega N \]
in the stable category. We will show that 
the map $\phi \colon N \lrar \Omega N$
has the desired properties.
Almost split sequences are, by definition, non-split short
exact sequences, and therefore the
boundary map $\phi$ in the above triangle
must be a non-trivial map in the stable category.

The next thing to be shown is that the map 
$\phi \colon N \lrar \Omega N$ induces the
zero map on the functors
$\uHom_{kG}(\Omega^i M , -)  \cong \hExt^i(M, -)$
for all $i$. To this end, consider any map
$f  \colon \Omega^i M  \lrar N$. We have to show that the
composite
\[ \Omega^i M  \lstk{f} N \lstk{\phi} \Omega N \]
is trivial in the stable category. Consider the following diagram
\[
\xymatrix{
 & & \Omega^i M \ar[d]^f \ar@{.>}[dl]  & \\
\Omega^2N \ar[r] & B \ar[r] & N \ar[r]^\phi & \Omega N
}
\]
where the bottom row is our distinguished triangle.
The map $f \colon  \Omega^iM \lrar  N$ cannot be a split epimorphism
by the given hypothesis, therefore by the
defining property of an almost split sequence,
the map $f$ factors through the middle term
$B$ as shown in the above diagram. Since the composite
of any two successive maps in a distinguished
triangle is zero, the composite $\phi \circ f$
is also zero by commutativity. So we are done.
\end{proof}

\begin{cor}\label{strategy}
If there is an indecomposable non-projective module $N$ in $thick_G(k)$
which is not isomorphic to $\Omega^i k$ for any $i$, then the GH fails
for $kG$.
\end{cor}

\begin{proof}
We apply the previous theorem with $M = k$. Since $\hExt_{kG}^*(k,N)
\cong \HHHH^*(G, N)$, the existence of an indecomposable 
nonprojective module $N$ in
$\thick_G(k)$ with the property that $N \not\cong \Omega^n k$ for
any $n$ implies (by the above theorem) the existence of a non-trivial map
$\phi: N \longrightarrow \Omega N$ in $\thick_G(k)$
such that the induced map of $\HHHH^*(G,k)$-modules,
$\HHHH^*(G,N) \longrightarrow \HHHH^*(G,N)$, is the zero map. In other
words $\phi \colon N \lrar \Omega N$ is a non-trivial ghost in $\thick(k)$.
Therefore the GH fails for $kG$.
\end{proof}

It is not hard to produce modules that satisfy the conditions laid out
in the above corollary. Specifically we have the following.

\begin{thm} \label{GHforgK}
Suppose that the  Sylow $p$-subgroup $P$ of $G$ has
order is at least 4. Then the GH fails for $kG$.
\end{thm}

\begin{proof}
We divide the proof into two cases. First suppose that $\HHH^*(G,k)$
is periodic, implying that $P$ is either cyclic or quaternion.
By Tate duality, there must exist a nonzero element $\zeta \in \HHHH^n(G,k)$
for some odd integer $n$. Then we have
an exact sequence
$$
\xymatrix{
\CE_{\zeta} \colon \qquad
0 \ar[r] & L_{\zeta} \ar[r] & \Omega^n k \ar[r]^{\zeta} & k \ar[r] & 0
}
$$
where $\zeta$ in the sequence is a cocycle representing the cohomology
element $\zeta$, and $L_{\zeta}$ is the kernel of $\zeta$. In the
case that $P$ is cyclic, $\Omega^n k$ when restricted to $P$ is a
direct sum of a projective module and a uniserial module of
dimension $\vert P \vert -1$. Because the cohomology element $\zeta$
is not zero, its restriction to $P$ is not zero and the restriction
of $L_{\zeta}$ is the direct sum of a projective module and a uniserial
module of dimension $|P|-2$. Consequently, $L_{\zeta}$ must be
indecomposable and moreover, because its dimension is not 1 or -1
modulo $|P|$,  $L_{\zeta}$ is not isomorphic to
$\Omega^n k$ for any $n$. So by corollary \ref{strategy}, the GH fails.

If $P$ is a quaternion group, then again by Tate duality, we can assume
that $n$ is positive and congruent to -1 modulo 4. When we restrict
$L_{\zeta}$ to $P$ we get the direct sum of a projective module and
a copy of $\Rad(\Omega^{-1} k)$ which is easily seen to be indecomposable.
So by the same argument as before, the GH fails.

Now suppose that the cohomology ring $\HHH^*(G,k)$ is not periodic.
This means that the maximal ideal spectrum 
$V_G(k)$ of the cohomology ring $\HHH^*(G,k)$ has Krull dimension
at least two. This time we choose $\zeta \in \HHH^n(G,k)$ with
the property that $n > 0$ and $\zeta$ not nilpotent. Then we construct
the sequence $\CE_{\zeta}$ and the module $L = L_{\zeta}$ exactly as in
the periodic case. The support variety of the module $L$ is
equal to $V_G(L) = V_G(\zeta)$, the collection of maximal ideals
that contain the element $\zeta$ \cite{Connected}. (Note here that we
do not need to assume that the field $k$ is algebraically closed,
though the proof of the statement about the variety of $L_{\zeta}$
requires extending the scalars to the algebraically closed case.)
Because $\zeta$ is not
nilpotent, $V_G(L)$ is a proper subvariety of $V_G(k)$ and the same
statement will hold for any direct summand of $L$. Hence, if $U$ is any
indecomposable direct summand of $L$, we must have that $U \not\cong
\Omega^n(k)$ for all $n$, simply because the support varieties are different.
Consequently by Corollary \ref{strategy}, the GH cannot
hold for $kG$.
\end{proof}

\section{Consequences of the GH}
We now derive some consequences of  the GH.  First of all, we
would like to point out that our main result implies that
the GH for $kG$ depends only on the characteristic of the
field $k$. In other words, if $k_1$ and $k_2$
are two fields of characteristic $p$ which divides the
order of $G$, then the GH holds for $k_1G$ if and only
if it holds for $k_2G$. It is not clear how one would prove
this fact directly.

 The dual generating hypothesis is also a natural problem
to ask. That is, instead of using the
(covariant) Tate cohomology functor $\uHom_{kG}(\Omega^*k , -)$,
we can use the (contravariant)
dual Tate cohomology functor $\uHom_{kG}(-, \Omega^*k)$,
and ask if this contravariant functor detects
trivial maps in $\thick_G(k)$.  The duality functor
$M \mapsto M^*$ sets a tensor triangulated equivalence between
$\thick_G(k)$ and its opposite category
$\thick_G(k)^{\text{opp}}$. In particular, the  GH holds for $kG$
if and only if the dual GH holds for $kG$.
Combining this with our previous results we have:

\begin{thm} \label{equiv}
The following assertions are equivalent:
\begin{enumerate}
\item The Sylow $p$-subgroup of $G$ is either $C_2$ or $C_3$.
\item The functor $\uHom_{kG}(\Omega^*k , -)$ is faithful on $\thick_G(k)$.
\item The functor $\uHom_{kG}( -, \Omega^*k)$ is faithful on  $\thick_G(k)$.
\item $\thick_G(k)$ consists of finite direct sums of suspensions of $k$.
\end{enumerate}
\end{thm}

\begin{proof}
$(1) \implies (2)$ is shown in \cite{CCM4}, 
$(2) \implies (4)$ is shown in Corollary
\ref{strategy}, $(4) \implies (2)$ is trivial, 
and $(2) \implies (1)$ is shown in Theorem 
\ref{GHforgK}. Finally the equivalence of statements 
(2) and (3) is shown in the paragraph
preceding this theorem.
\end{proof}

Our final result is motivated by a result of Peter 
Freyd \cite{freydGH} which states that
if the stable homotopy functor on the category of
finite spectra is faithful then it is
also full. We now prove the analogue of this statement
for the stable module category.  This
generalizes Theorem 3.3 of \cite{CCM3} where we
established the same result for $p$-groups.

\begin{thm}
If the GH holds for $kG$, then the Tate cohomology
functor from $\thick_G(k)$ to the category of modules
over $\HHHH^*(G, k)$ is also full.
\end{thm}

\begin{proof}
If the GH holds for $kG$, then from the 
equivalence ($2 \iff 4$) of Theorem \ref{equiv} we know that 
$\thick_G(k)$ is made up of finite  direct sums of suspensions of $k$.
In particular, for each $M$ in $\thick_G(k)$,
$\HHHH^*(G, M)$  is a free $\HHHH^*(G, k)$-module of finite rank.
It follows that the induced map
\[ \uHom_{kG}(M, X) \lrar \Hom_{\HHHH^*(G, k)}
(\HHHH^*(G, M), \HHHH^*(G, X)) \]
is an isomorphism for all $kG$-modules $X$.
Since $M$ was an arbitrary
$kG$-module in $\thick_G(k)$, we have shown
that the functor $\HHHH^*(G, -)$
is full, as desired.
\end{proof}

The class of groups for which the GH has an affirmative answer (namely,
groups whose Sylow $p$-subgroup is $C_2$ or $C_3$) although small, has
some interesting candidates. This includes 
interesting simple groups. For example,
it has the smallest simple group $A_5$ of order $60$ in characteristic $3$.
Another  interesting example in characteristic $3$ is the smallest Janko group
$J_1$  which has  order $175560$. In fact, $J_1$ is
the only sporadic simple group for which the GH holds.

\section{Symmetric algebras}
We end the paper by pointing out that much of the work 
here applies for any finite dimensional
symmetric algebra. Note that if $A$ is such an algebra, we have almost 
split sequences ending in any finitely generated 
indecomposable non-projective $A$-module $N$:
\[ 0 \lrar \Omega^2 N \lrar E \lrar N \lrar 0.\]
Therefore our proofs generalize immediately to give the following result.

\begin{thm}  Let $A$ be a finite dimensional symmetric 
$k$-algebra. Fix a finitely
generated non-projective indecomposable $A$-module $L$. 
Then the following are equivalent.
\begin{enumerate}
\item The functor $\uHom_A(\Omega^* L, -)$ is faithful 
on $\thick(L)$, the thick subcategory
   generated by $L$ in $\stmod(A)$.
\item $\thick(L)$ consists of modules isomorphic to 
finite direct sums of suspensions of $L$.
\end{enumerate}
\end{thm}

We sketch a proof and leave the details to the reader.

\begin{proof}
Clearly (2) implies (1). For the other direction, 
suppose that (2) fails. Then we have a
an indecomposable non-projective object $N$ in 
$\thick(L)$  such that $N \ncong \Omega^i L$ for any $i$.
Then the map $N \lrar \Omega N$ corresponding to 
the almost split sequence ending in $N$ is shown 
(exactly as before) to be a non-trivial map that 
is invisible to the functor $\uHom_A(\Omega^*L, -)$. 
This shows that (1) fails. So we are done.
\end{proof}

\end{document}